\documentclass[12pt, a4paper]{article}
\usepackage[latin1]{inputenc}
\usepackage[english]{babel}
\usepackage{amsmath,verbatim,amsfonts,amsthm,amssymb,mathrsfs,stmaryrd}
\usepackage{geometry}
\usepackage{graphicx}
\immediate\write18{texcount -tex -sum  \jobname.tex > \jobname.wordcount.tex}

\usepackage{ifthen}

\usepackage[bookmarksnumbered]{hyperref}

\usepackage{eucal}	
\usepackage{tensor}
\usepackage[dvipsnames]{xcolor}
\usepackage{csquotes}
\usepackage{tabularx}
\usepackage{blindtext}
\usepackage{multicol}
\usepackage{enumitem}
\usepackage[normalem]{ulem} 
\usepackage{tikz}  
\usetikzlibrary{arrows,positioning}

\usepackage[all]{xy}
\usepackage{amsfonts} 
\usepackage{textcomp}
\usepackage{color}
\graphicspath{ {./images/} }
\usepackage[noabbrev, capitalise, nameinlink, nosort]{cleveref}

\usepackage[T1]{fontenc}
\usepackage{calligra}

\newcounter{generalnumbering}   \numberwithin{generalnumbering}{section}

\theoremstyle{plain}    \newtheorem{theorem}[generalnumbering]{Theorem}
\theoremstyle{plain}    
\theoremstyle{definition}   \newtheorem{definition}[generalnumbering]{Definition}
\theoremstyle{definition}   \newtheorem{example}[generalnumbering]{Example}
\theoremstyle{plain}    \newtheorem{proposition}[generalnumbering]{Proposition}
\theoremstyle{plain}    \newtheorem{lemma}[generalnumbering]{Lemma}
\theoremstyle{definition}   \newtheorem{remark}[generalnumbering]{Remark}

\newcommand{\namefordifferentenvironment}{}

\theoremstyle{plain}    \newtheorem{plainstyle}[generalnumbering]{\namefordifferentenvironment}
\theoremstyle{plain}    \newtheorem*{plainstyle*}{\namefordifferentenvironment}
\theoremstyle{definition}    \newtheorem{definitionstyle}[generalnumbering]{\namefordifferentenvironment}
\theoremstyle{definition}    \newtheorem*{definitionstyle*}{\namefordifferentenvironment}

\newenvironment{penv*}[1]{\renewcommand{\namefordifferentenvironment}{#1}\begin{plainstyle*}}{\end{plainstyle*}}

\newenvironment{denv*}[1]{\renewcommand{\namefordifferentenvironment}{#1}\begin{definitionstyle*}}{\end{definitionstyle*}}

\DeclareMathOperator{\Se}{\mathcal{S}}

\DeclareMathOperator{\id}{id}

\newcommand\restr[2]{\ensuremath{\left.#1\right|_{#2}}}

\title{\Large Globalization of partial inverse semigroupoid actions on sets}
\author{Paulinho Demeneghi\footnote{Universidade Federal de Santa Catarina.}, Felipe Augusto Tasca\footnote{Instituto Federal do Paran\'{a}.}}

\date{}

\begin{document}

\maketitle

\begin{abstract}
We prove that every partial action of an inverse semigroupoid on a set admits a universal globalization. Moreover, we show that our construction gives a reflector from the category of partial actions on the full subcategory of global actions. Finally, we investigate if the mediating function given by the universal property of our construction is injective.

\end{abstract}

\let\thefootnote\relax\footnotetext{The authors report there are no competing interests to declare.}


\section{Introduction}

The notion of partial action has appeared for the first time in the literature in \cite{exelorigem}, in a work developed by Exel in the context of C*-algebras where the concept of crossed product of a C*-algebra by a partial action of the infinite cyclic group was introduced. The work of Exel was generalized afterward by McClanahan in \cite{mcclanahan} where a formal definition of crossed product of a C*-algebra by a partial action of a discrete group was presented. Since then, partial actions of groups were extensively explored and studied in many other areas.

The notion of partial action of groups on sets was introduced by Exel in \cite{exel1998} motivated by the great interest in the concept of partial actions of groups on C*-algebras. Since then, the concept of group partial actions have been appeared in many  contexts.

Global actions naturally induce partial actions via a restriction process and it is an important problem to know when a partial action can be obtained through the restriction of a global one. Explicitly, given a partial action of a group $G$ on a set $X$, one can obtain a new action by restricting the action of $G$ to a subset $Y \subseteq X$. The resultant action is a partial action of $G$ on $Y$ which is not, in general, a global action.

On the other side, given a partial action of $G$ on a set $Y$, one may ask if there exists a global action of $G$ on a set $X$ containing $Y$ such that the original partial action of $G$ on $Y$ can be obtained by the restriction of the global action of $G$ on $X$. This question motivates the notion of globalization of partial actions. A globalization of a partial action $\theta$ of a group $G$ on a set $Y$ is (isomorphic to) a global action $\tilde{\theta}$ of $G$ on a set $X \supseteq Y$ such that $\theta$ may be obtained by the restriction of $\tilde{\theta}$ to $Y$.

The problem of the existence of globalizations was initially studied by Abadie in \cite{Abadie,Abadie2} and later by Kellendonk and Lawson in \cite{kellendonklawson} in an independent way. In \cite{Abadie}, Abadie proved that every partial action of a group on a set admits a universal globalization. In this same paper, Abadie has also studied enveloping actions of groups on topological spaces and C*-algebras. In \cite{kellendonklawson}, in addition to the context of sets and topological spaces, Kellendonk and Lawson also studied partial action on semi-lattices.

Motivated by the problem of the existence of globalizations for partial actions, many works were developed in many other contexts exploring: partial actions of groups on objects like topological spaces, rings, C*-algebras; partial actions of monoids, semigroups, inverse semigroups, Hopf algebras, ordered groupoids, groupoids and small categories; and even contexts involving twisted actions. We may cite \cite{eliezer, bagio, Dirceu, thaisa2021, dokuchaev2, dokuchaev, Gilbert, mykola, ganna, megrelis, Patrik, steinberg} as example. Furthermore, we indicate Dokuchaev extensive survey on the subject \cite{dokuchaevsurvey}.

In this paper we are interested in studying the problem of globalization of inverse semigroupoid partial actions on sets. The notion of semigroupoids includes both the notions of categories and semigroups. A semigroupoid may be thought as a set equipped with a partially defined associative operation. Actually, there are two notions of semigroupoids in the literature and both notions are obtained by weakening category axioms. The first notion of semigroupoid is presented by Tilson in \cite{tilson} in order to treat category ideals independently of their parent categories. Roughly speaking, a Tilson semigroupoid is a small category which may fail to have identity arrows. The second notion is presented by Exel in \cite{exelsemigroupoid1} with the goal of providing a unified treatment for C*-algebras associated to combinatorial objects. As a matter of fact, Exel's notion is more general than the Tilson's one, since Tilson semigroupoids comes with an underlying graph structure, peculiar from categories, which completely characterizes which pairs of morphisms can be composed, meanwhile Exel semigroupoids comes equipped with a set of composable pairs which may not admit a compatible underlying graph structure. Roughly speaking, in Exel semigroupoids it may not be possible to compose morphisms even though the codomain of the first agrees with the domain of the second.

The inverse semigroupoids framework provides a unified treatment for inverse semigroups and groupoids and, it turns out that, despite the difference between the two existing notions of semigroupoids, when the unique pseudo-inverse axiom is added, every Exel inverse semigroupoid can be realized as a Tilson inverse semigroupoid in a unique way, as proved by Cordeiro in \cite{Luiz1}. So that, for a better fit in our approach, we have opted by Tilson's definition.

Partial actions of inverse semigroupoids were introduced by Cordeiro in \cite{Luiz2, Luiz1}, where an extensive study on semigroupoids is developed. Based on groupoid partial actions, in this work we present an adaptation of this definition, which is equivalent to Cordeiro's definition in this context.

We prove that, just as group partial actions on sets, every inverse semigroupoid partial action on a set admits a universal globalization. This result generalizes simultaneously Theorem 4.23 of \cite{Abadie} for group partial actions on sets, Theorem 6.10 of \cite{hollings} for inverse semigroups partial actions on sets and Theorem 4 (Remark 29) of \cite{Patrik} for groupoids partial actions on sets. Moreover, we prove that our constructions gives a reflector from the category of partial actions on the full subcategory of global actions.

The paper is structured as follows: in the second section we first present the definitions of semigroupoid and inverse semigroupoid adopted, as well as basics facts about these structures that we are going to need. Then we proceed to define inverse semigroupoid partial actions on sets and introduce the necessary machinery for the development of this work. In the third section we fix a partial action $\theta$ of an inverse semigroupoid $\Se$ on a set $X$ and then construct a globalization of $\theta$ which we further prove that is a reflector of $\theta$ in the full subcategory of global actions (Theorem \ref{o_teorema}) and, hence, is universal among all globalizations of $\theta$ (Theorem \ref{o_teorema_2}). Finally, we explore some examples, one of them, showing that the mediating function given by the universal property of our construction may not be injective. However, we show that it is injective on some special subsets (Proposition \ref{a_proposicao}).


\section{Inverse Semigroupoids and Partial Actions}

In this section we present the necessary machinery involving inverse semigroupoids and inverse semigroupoid partial actions on sets. For further references, we recommend \cite{Luiz1, liu}.

We first present the definition of semigroupoid we shall adopt.

\begin{definition}\label{def_semigroupoid}
A \emph{semigroupoid} is a quintuple
$(\Se,\Se^{(0)},d,c,\mu)$, in which
\begin{itemize}
    \item $\Se$ and $\Se^{(0)}$ are sets called \emph{arrow} and \emph{object} sets, respectively;
    \item $d,c: \Se \to \Se^{(0)}$ are maps called \emph{domain} and \emph{codomain} maps, respectively;
    \item $\mu: \Se^{(2)} \to \Se$ is a map called \emph{multiplication} map (denoted by juxtaposition $\mu(s,t)=st$) defined in the set $\Se^{(2)}\subseteq \Se \times \Se$ of \emph{composable pairs} consisting of all pairs $(s,t) \in \Se \times \Se$ such that $d(s)=c(t)$.
\end{itemize}
satisfying
\begin{enumerate}[label=(\roman*)]
    \item $d(st)=d(t)$ and $c(st)=c(s)$ for any composable pair $(s,t) \in \Se^{(2)}$;
    \item $(ps)t=p(st)$ whenever $(p,s), (s,t) \in \Se^{(2)}$.
\end{enumerate}
\end{definition}

Notice that, if $p,s,t \in \Se$ are elements such that $(p,s)$ and $(s,t)$ lies in $\Se^{(2)}$, then by axiom (i) above $(ps,t)$ and $(p,st)$ also lies in $\Se^{(2)}$, and hence axiom (ii) does make sense.

By an abuse of notation, whenever there is no chance of confusion, we shall use just $\Se$ to refer to semigroupoid $(\Se,\Se^{(0)},d,c,\mu)$. 

Notice that the class of semigroupoids includes semigroups and small categories. Semigroups are precisely the semigroupoids with a single object (and not necessarily an identity element) and small categories are precisely the semigroupoids containing one identity for each object.

To be precise, by an identity element we mean an element $e \in \Se$ in a semigroupoid $\Se$ satisfying $es=s$ and $te=t$ for every $s,t \in \Se$ such that $(e,s), (t,e) \in \Se^{(2)}$. If necessary, we may say that $e \in \Se$ is an identity over an object $u \in \Se^{(0)}$ to emphasize that $e$ is an identity element such that $d(e)=c(e)=u$.

Another sort of elements which play an essential role in this work are idempotent elements. By an idempotent element in a semigroupoid $\Se$, we mean an element $e \in \Se$ such that $(e,e) \in \Se^{(2)}$ and $ee=e$. Again, we may emphasize that $d(e)=c(e)=u$ by saying that $e$ is an idempotent over $u \in \Se^{(0)}$. We denote by $E(\Se)$ the set of idempotent elements of $\Se$, which evidently includes the identity elements.

As mentioned before, we are mainly interested in inverse semigroupoids, as follows.

\begin{definition}\label{def_inv_semigroupoid}
A semigroupoid $\Se$ is an \emph{inverse semigroupoid} if for every $s \in \Se$ there exists a unique $s^* \in \Se$, called the \emph{inverse} of $s$, such that $(s,s^*), (s^*,s) \in \Se^{(2)}$ and
$$ss^*s=s \quad \text{ and } \quad s^*ss^*=s^*.$$
\end{definition}

Just like the class of semigroupoids includes small categories and semigroups, the class of inverse semigroupoids includes groupoids and inverse semigroups.

We shall see now that many inverse semigroup properties keep working for this more general setting. For example, usual rules for inverses are preserved: $(s^*)^*=s$ for every $s \in \Se$ and if $(s,t) \in \Se^{(2)}$ then $(t^*,s^*) \in \Se^{(2)}$ and $(st)^*=t^*s^*$. For every element $s \in \Se$ we have  $s^*s, ss^* \in E(\Se)$ and for every $e \in E(\Se)$ we have $e^*=e$. Moreover, the idempotent elements commute in the sense that, if $e,f \in E(\Se)$ and $(e,f) \in \Se^{(2)}$, then $(f,e) \in \Se^{(2)}$ and $ef=fe$. This implies that for every $s \in \Se$ and $e \in E(\Se)$ such that $(s,e) \in \Se^{(2)}$ we have $ses^* \in E(\Se)$.

The commutativity of $E(\Se)$ enables us to define a canonical order relation on any inverse semigroupoid $\Se$ just like in the inverse semigroup case. More precisely, if $s,t \in \Se$, we say that $s \leq t$ if $d(s)=d(t)$, $c(s)=c(t)$ and one of the following equivalent statements holds:
\begin{enumerate}[label=(\roman*)]
    \item $s=ts^*s$;
    \item $s=te$ for some $e\in E(\mathcal{S})$ such that $(t,e) \in \Se^{(2)}$;
    \item $s=ss^*t$;
    \item $s=ft$ for some $f\in E(\mathcal{S})$ such that $(f,t) \in \Se^{(2)}$.
\end{enumerate}

Notice that, if $s,t \in \Se$, then $s \leq t$ if and only if $s^* \leq t^*$ and that, if $s_1, s_2, t_1, t_2 \in \Se$ are such that $(s_1,s_2) \in \Se^{(2)}$, $s_1\leq t_1$ and $s_2\leq t_2$, then $(t_1,t_2) \in \Se^{(2)}$ and $s_1s_2\leq t_1t_2$. Moreover, it is immediate from the definition that, if $e,f \in E(\Se)$ and $(e,f)\in \Se^{(2)}$, than $ef\leq e,f$. As a consequence of the latter, $ses^*\leq ss^*$ whenever $s \in \Se$, $e \in E(\Se)$ and $(s,e) \in \Se^{(2)}$.

Before we proceed to the definition of partial action, we present a inverse semigroupoid example which is not an inverse semigroup, nor a groupoid (not even a category).

\begin{example}\label{ex_pt1}
As a set, $\Se$ is given by $\Se=\{a,a^*,b,b^*,a^*a,aa^*,b^*b,bb^*\}$. The object set is a two element set and the graph structure may be viewed in the following picture.

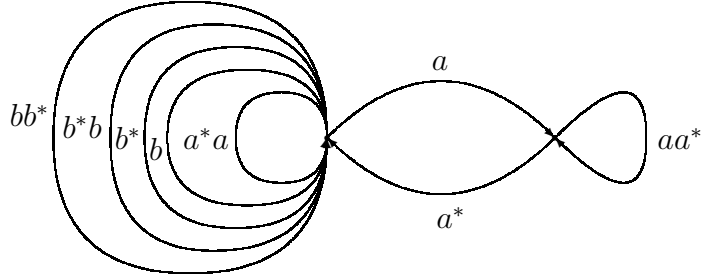
\begin{center}
\setlength{\unitlength}{3mm}
\begin{picture}(10,10)
    \put(2,5){\circle*{0.2}}
    \put(12,5){\circle*{0.2}}
				
    \qbezier(2,5)(7,10)(12,5)
    \put(12,5){\vector(1,-1){0}}
		    
    \qbezier(2,5)(7,0)(12,5)
    \put(2,5){\vector(-1,1){0}}

    \qbezier(2,5)(2,3)(0,3)
    \qbezier(0,3)(-2,3)(-2,5)
    \qbezier(2,5)(2,7)(0,7)
    \qbezier(0,7)(-2,7)(-2,5)
    \put(2,5){\vector(1,4){0}}
    \put(2,5){\vector(0,1){0}}
    
    \qbezier(2,5)(2,2)(-1.5,2)
    \qbezier(-1.5,2)(-5,2)(-5,5)
    \qbezier(2,5)(2,8)(-1.5,8)
    \qbezier(-1.5,8)(-5,8)(-5,5)
    
    \qbezier(2,5)(2,1)(-2,1)
    \qbezier(-2,1)(-6,1)(-6,5)
    \qbezier(2,5)(2,9)(-2,9)
    \qbezier(-2,9)(-6,9)(-6,5)
    
    \qbezier(2,5)(2,0)(-3,0)
    \qbezier(-3,0)(-7.5,0)(-7.5,5)
    \qbezier(2,5)(2,10)(-3,10)
    \qbezier(-3,10)(-7.5,10)(-7.5,5)
    
    \qbezier(2,5)(2,-1)(-4,-1)
    \qbezier(-4,-1)(-10,-1)(-10,5)
    \qbezier(2,5)(2,11)(-4,11)
    \qbezier(-4,11)(-10,11)(-10,5)
   
    \qbezier(12,5)(16,1)(16,5)
    \qbezier(12,5)(16,9)(16,5)
    \put(12,5){\vector(-1,1){0}}

    \put(6.6,8){$a$}
    \put(6.8,1){$a^*$}
    \put(-4.3,4.5){$a^*a$}
    \put(-5.8,4){$b$}
    \put(-7.3,4.5){$b^*$}
    \put(-9.6,5){$b^*b$}
    \put(-11.9,5.5){$bb^*$}
    \put(16.5,4.5){$aa^*$}
\end{picture}
\end{center}

Finally, the multiplication table is given by:
\begin{center}
\begin{tabularx}{0.8\textwidth}{
  | >{\centering\arraybackslash}X 
  | >{\centering\arraybackslash}X 
  | >{\centering\arraybackslash}X 
  | >{\centering\arraybackslash}X 
  | >{\centering\arraybackslash}X 
  | >{\centering\arraybackslash}X 
  | >{\centering\arraybackslash}X
  | >{\centering\arraybackslash}X 
  | >{\centering\arraybackslash}X |}
 \hline
 $\cdot$ & $a$ & $a^*$ & $b$ & $b^*$ & $a^*a$ & $aa^*$ & $b^*b$ & $bb^*$  \\
 \hline
 $a$  & $-$  & $aa^*$ & $a$ & $a$ & $a$  & $-$ & $a$ & $a$\\ 
\hline
$a^*$  & $a^*a$  & $-$ & $-$ & $-$ & $-$  & $a^*$ & $-$ & $-$\\
\hline
 $b$  & $-$  & $a^*$ & $a^*a$ & $bb^*$ & $a^*a$  & $-$ & $b$ & $a^*a$\\ 
 \hline
$b^*$  & $-$  & $a^*$ & $b^*b$ & $a^*a$ & $a^*a$  & $-$ & $a^*a$ & $b^*$\\ 
\hline
$a^*a$  & $-$  & $a^*$ & $a^*a$ & $a^*a$ & $a^*a$  & $-$ & $a^*a$ & $a^*a$\\ 
\hline
 $aa^*$  & $a$  & $-$ & $-$ & $-$ & $-$  & $aa^*$ & $-$ & $-$\\
\hline
 $b^*b$  & $-$  & $a^*$ & $a^*a$ & $b^*$ & $a^*a$  & $-$ & $b^*b$ & $a^*a$\\ 
 \hline
$bb^*$  & $-$  & $a^*$ & $b$ & $a^*a$ & $a^*a$  & $-$ & $a^*a$ & $bb^*$\\ 
\hline
\end{tabularx}
\end{center}
in which the symbol ``$-$'' means that the product is undefined.

Notice that, the idempotent set is given by $E(\Se)=\{a^*a, aa^*, b^*b, bb^*\}$.
\end{example}

We now proceed to present the definition of a partial action of an inverse semigroupoid on a set.

\begin{definition}\label{our_def}  
A \emph{partial action} $\theta$ of an inverse semigroupoid $\mathcal{S}$ on a set $X$ is a pair
$$\theta=\big(\{X_s\}_{s \in \Se}, \{\theta_s\}_{s \in \Se}\big)$$
consisting of a collection $\{X_s\}_{s \in \Se}$ of subsets of $X$ and a collection $\{\theta_s\}_{s \in \Se}$ of maps
$$\theta_s \colon X_{s^*} \to X_s$$
such that
\begin{enumerate}[label=(P\arabic*)]
    \item\label{p1} $\theta_e=\id_{X_e}$ for all $e \in E(\Se)$. Moreover, for all $x \in X$, there exists $e \in E(\Se)$ such that $x \in X_e$;
    \item\label{p2} $X_s \subseteq X_{ss^*}$ for every $s \in \Se$;
    \item\label{p3} $\theta_t^{-1}(X_t \cap X_{s^*})=X_{(st)^*}\cap X_{t^*}$ for all $(s,t)\in\mathcal{S}^{(2)}$ and, moreover, $\theta_s\big(\theta_t(x)\big)=\theta_{st}(x)$ for all $x \in X_{(st)^*}\cap X_{t^*}$.
\end{enumerate}

We say that $\theta$ is a \emph{global action} if moreover it satisfies: 
\begin{enumerate}[label=(P\arabic*)]\setcounter{enumi}{3}
    \item\label{pglob} $X_{s} = X_{ss^*}$ for every $s \in \Se$ (i.e., equality holds in \ref{p2}).
\end{enumerate}

By an $\Se$-\emph{set} we shall mean a pair $(X,\theta)$ in which $X$ is a set, $\Se$ is an inverse semigroupoid and $\theta$ is a partial action of $\Se$ on $X$. We shall, moreover, say that $(X,\theta)$ is a \emph{global} $\Se$-\emph{set} if $\theta$ is a global action.
\end{definition}

We should also point out that, for an $\Se$-\emph{set} $(X,\theta)$, unless explicitly said otherwise, we will use the same symbols $X$ and $\theta$, indexed by elements of $\Se$, to denote the members of the family of subsets of $X$ and the members of the family of maps that constitute $\theta$, respectively. That is, we will implicitly assume that $\theta$ is given by $(\{X_s\}_{s \in \Se}, \{\theta_s\}_{s \in \Se})$.

Before we proceed, let us point some observations about the last definition.

\begin{remark}\label{obs1}
First, notice that $\bigcup_{e\in E(\Se)} X_e=X$ by \ref{p1}. This is not a very restrictive requirement since, in its absence, we may replace $X$ by the union of the domains $X_e$, leading to an action of $\Se$ on $\bigcup_{e\in E(\Se)} X_e$ satisfying all the desired properties. Actually, some authors refer to partial actions with this extra requirement as \emph{non-degenerate} partial actions. We have chosen to not follow this approach and ask for it directly in the definition.

Replacing $s$ by $s^*$ in \ref{p2}, we deduce that the domain $X_{s^*}$ of $\theta_s$ is contained in the domain $X_{s^*s}$ of $\theta_{s^*s}$ and, moreover, by \ref{pglob}, $\theta_s$ and $\theta_{s^*s}$ share domains if the action is global.

Given $s,t \in \Se$, we may not be able to compose $\theta_s$ and $\theta_t$ in the traditional way, since the image of $\theta_t$ may not be contained in the domain of $\theta_s$. However, we will use the notation $\theta_s\circ\theta_t$ to refer to the map defined on the largest domain in which the expression $\theta_s(\theta_t(x))$ does make sense. Explicitly, the domain and codomain of $\theta_s\circ\theta_t$ are considered to be $\theta_t^{-1}(X_t\cap X_{s^*})$ and $\theta_s(X_{s^*}\cap X_t)$, respectively.

Hence, under the convention explained in the last paragraph, we can deduce from \ref{p3} that, whenever $(s,t)\in\Se^{(2)}$, the domain $\theta_t^{-1}(X_t\cap X_{s^*})$ of $\theta_s\circ\theta_t$ is contained in the domain $X_{(st)^*}$ of $\theta_{st}$ and $\theta_s\circ\theta_t$ coincides with $\theta_{st}$ in the domain of $\theta_s\circ\theta_t$. In other words, from \ref{p3} we deduce that $\theta_{st}$ is an extension\footnote{We shall use the term extension even though the codomain of the functions involved may not agree with each other.} of $\theta_s\circ\theta_t$ for $(s,t)\in\Se^{(2)}$, which we shall denote by $\theta_s\circ\theta_t\subseteq\theta_{st}$. That is, we shall use the symbol ``$\subseteq$'' to express that the function on the right-hand side is an extension of the one in the left-hand side.
\end{remark}

\begin{proposition}\label{inverse}
Let $\Se$ be an inverse semigroupoid and $(X,\theta)$ an $\Se$-set. For every $s \in \Se$, $\theta_s$ is a bijection from $X_{s^*}$ onto $X_s$ and, moreover, $\theta_s^{-1}=\theta_{s^*}$.
\end{proposition}

\begin{proof}
Let $s \in \Se$. By \ref{p3} we deduce that $\theta_{s^*s}$ is an extension of $\theta_{s^*}\circ\theta_s$. By \ref{p1}, it follows that $\theta_{s^*s}$ is the identity map on $X_{s^*s}$. Hence, $\theta_{s^*}\circ\theta_s$ is the identity map on its domain, which is clearly $X_{s^*}$. Similarly, $\theta_s\circ\theta_{s^*}$ is the identity map on $X_s$.
\end{proof}

With this result in hands we can rewrite the domain $\theta_t^{-1}(X_t \cap X_{s^*})$ of $\theta_s\circ\theta_t$ as $\theta_{t^*}(X_t \cap X_{s^*})$. This enables us to obtain a characterization for the codomain $\theta_s(X_{s^*}\cap X_t)$ of $\theta_s\circ\theta_t$ in the same way \ref{p3} provides one for the domain of $\theta_s\circ\theta_t$.

\begin{proposition}\label{ran_composition}
Let $\Se$ be an inverse semigroupoid and $(X,\theta)$ an $\Se$-set. If $(s,t)\in\Se^{(2)}$, then
$$\theta_s(X_{s^*}\cap X_t)=X_s\cap X_{st}.$$
\end{proposition}

\begin{proof}
Let $(s,t)\in\Se^{(2)}$. Hence, $(t^*,s^*)\in\Se^{(2)}$ and, by \ref{p3}, $\theta_{s^*}^{-1}(X_{s^*} \cap X_t)=X_{st}\cap X_s$. The result now follows from the comment immediately before the statement.
\end{proof}

We now compare $\theta_s$ and $\theta_t$ when $s \leq t$.

\begin{proposition}\label{extension}
Let $\Se$ be an inverse semigroupoid and $(X,\theta)$ an $\Se$-set. If $s\leq t$, then $X_s \subseteq X_t$ and $\theta_s \subseteq \theta_t$.
\end{proposition}

\begin{proof}
Let $s,t \in \Se$ such that $s\leq t$. Since $s\leq t$, we have that $s=ss^*t$. Let $x \in X_s$. First, we shall prove $X_s \subseteq X_t$. By \ref{p2}, $x \in X_{ss^*}$. Hence, by \ref{ran_composition} and \ref{p1}, we have
$$x \in X_{ss^*}\cap X_s=X_{ss^*}\cap X_{ss^*t} \overset{\ref{ran_composition}}{=} \theta_{ss^*}(X_{ss^*}\cap X_t) \overset{\ref{p1}}{=} X_{ss^*}\cap X_t \subseteq X_t.$$
For the extension, we also have $s^*\leq t^*$ and, analogously to the previous argument, we have $X_{s^*}\subseteq X_{t^*}$. Then, for every $x \in X_{s^*}$, by \ref{p1} and \ref{p3}, we have that $x$ lies in the domain $X_{s^*}\cap X_{t^*}$ of $\theta_{ss^*}\circ\theta_t$ and
$$\theta_s(x)\overset{\ref{p3}}{=}\theta_{ss^*}\big(\theta_t(x)\big)\overset{\ref{p1}}{=}\theta_t(x).$$

Hence, $\theta_t$ is an extension of $\theta_s$ as stated.
\end{proof}

If $e,f$ are composable idempotent elements, the next result shows that $\theta_e\circ\theta_f$ and $\theta_{ef}$ share domains and, hence, coincide.

\begin{proposition}\label{ef_domain}
Let $\Se$ be an inverse semigroupoid and $(X,\theta)$ an $\Se$-set. If $e,f\in E(\Se)$ and $(e,f)\in\Se^{(2)}$, then $X_{ef}=X_e \cap X_f$ and $\theta_{ef}=\theta_e\circ\theta_f$.
\end{proposition}

\begin{proof}
Let $e,f \in E(\Se)$ such that $(e,f)\in S^{(2)}$. Since, $ef\leq e$, by \ref{extension}, we have $X_{ef}\subseteq X_e$. Hence, by \ref{p1} and \ref{ran_composition}, we have
$$X_{ef}=X_e \cap X_{ef} \overset{\ref{ran_composition}}{=} \theta_e(X_e \cap X_f) \overset{\ref{p1}}{=} X_e \cap X_f.$$
Since the domains (and codomains) of $\theta_{ef}$ and $\theta_e\circ\theta_f$ coincide, we must have $\theta_{ef}=\theta_e\circ\theta_f$, as desired.
\end{proof}

We now present one alternative for the definition of partial action. The reader is invited to compare the next proposition with Definition 1.7~\cite{Luiz2}.

\begin{proposition}\label{def_lui}
Let $X$ be a set and $\Se$ an inverse semigroupoid. Let $\{X_s\}_{s\in\Se}$ be a collection of subsets of $X$ and $\{\theta_s\}_{s\in\Se}$ a collection of maps $\theta_s\colon X_{s^*}\to X_s$. Then, the pair $\theta=(\{X_s\}_{s\in\Se}, \{\theta_s\}_{s\in\Se})$ is a partial action of $\Se$ on $X$ if, and only if, it satisfies
\begin{enumerate}[label=(E\arabic*)]
    \item\label{e1} For every $s\in\mathcal{S}$, $\theta_s$ is a bijection and $\theta_s^{-1}=\theta_{s^*}$. Moreover,
    $X=\cup_{s\in\Se} X_s$; 
    \item\label{e2} $\theta_s\circ\theta_t \subseteq \theta_{st}$ for every $s, t \in \Se$ such that $(s,t)\in\mathcal{S}^{(2)}$;
    \item\label{e3} $X_s \subseteq X_t$ for every $s,t \in \Se$ such that $s\leq t$.
\end{enumerate}
Furthermore, if $\theta$ is a partial action of $\Se$ on $X$, then it is a \emph{global action} if, and only if
\begin{enumerate}[label=(E\arabic*)]\setcounter{enumi}{3}
    \item\label{eglob} $\theta_{st}=\theta_s\circ\theta_t$ for every $s, t \in \Se$ such that $(s,t)\in\mathcal{S}^{(2)}$ (i.e. equality holds in \ref{e2}).
\end{enumerate}
\end{proposition}

\begin{proof}
If $\theta$ is a partial action, the first part of \ref{e1} follows from Proposition \ref{inverse} and the second part follows immediately from the second part of \ref{p1}. From the comments after Definition \ref{our_def} we deduce \ref{e2}, and \ref{e3} is a direct consequence of Proposition \ref{extension}.

Assume now that $\theta$ satisfies \ref{e1}, \ref{e2} and \ref{e3}. Our goal is to prove that $\theta$ is a partial action of $\Se$ on $X$.

By \ref{e1} and \ref{e2} we deduce \ref{p2} since
$$\id_{X_s} \overset{\ref{e1}}{=} \theta_s \circ \theta_{s^*} \overset{\ref{e2}}{\subseteq} \theta_{ss^*},$$
and, hence, the domain $X_s$ of $\id_{X_s}$ is contained in the domain $X_{ss^*}$ of $\theta_{ss^*}$. This also enables us to infer the second part of \ref{p1} since, by \ref{e1}, for any $x \in X$, there exists $s \in \Se$ such that $x \in X_s$. Hence, $x \in X_s \subseteq X_{ss^*}$ and $ss^*$ is an idempotent.

For the first part of \ref{p1} we use again \ref{e1} and \ref{e2}. In one hand we have 
$$\id_{X_e} \overset{\ref{e1}}{=} \theta_e\circ\theta_e \overset{\ref{e2}}{\subseteq} \theta_e$$
and, since $\id_{X_e}$ and $\theta_e$ share domains, we must have $\theta_e=\id_{X_e}$.

It remains to argue \ref{p3}. For that task, let $(s,t)\in \Se^{(2)}$. By \ref{e2}, $\theta_s\circ\theta_t\subseteq\theta_{st}$ and, hence, the domain $\theta_t^{-1}(X_t \cap X_{s^*})$ of $\theta_s\circ\theta_t$ is contained in the domain $X_{(st)^*}$ of $\theta_{st}$. Combining this with the fact that the domain of $\theta_t$ is $X_{t^*}$ we obtain
\begin{equation}\label{def_lui_1}
\theta_t^{-1}(X_t \cap X_{s^*}) \subseteq X_{(st)^*} \cap X_{t^*}.
\end{equation}
To deduce the reverse inclusion, notice that we must also have $(st,t^*)\in\Se^{(2)}$ and using \ref{e2} and a similar argument we obtain
\begin{equation}\label{def_lui_2}
\theta_{t^*}^{-1}(X_{t^*} \cap X_{(st)^*}) \subseteq X_{(stt^*)^*} \cap X_t \overset{\ref{e3}}{\subseteq} X_{s^*} \cap X_t,
\end{equation}
in which the last inclusion follows from \ref{e3}, since $(stt^*)^*=tt^*s^*\leq s^*$. Combining the fact that $X_{s^*} \cap X_t$ is contained in the domain of $\theta_{t^*}$ and \ref{e1}, we obtain from (\ref{def_lui_2}) that
\begin{equation}\label{def_lui_3}
X_{t^*} \cap X_{(st)^*} \subseteq \theta_{t}^{-1}(X_{s^*}\cap X_t)
.\end{equation}

Joining (\ref{def_lui_1}) and (\ref{def_lui_3}) we obtain the desired equality.
Finally, if $x \in X_{(st)^*}\cap X_{t^*}$, then $x$ lies in the domain $\theta_t^{-1}(X_t \cap X_{s^*})$ of $\theta_s\circ\theta_t$ and, hence, by \ref{e2}, we have that
$$\theta_s\big(\theta_t(x)\big)=\theta_{st}(x),$$
completing the verification of \ref{p3}.

Suppose now that $\theta$ is a global action of $\Se$ on $X$. Hence, by \ref{pglob} and Proposition \ref{ef_domain}, we have
$$X_{(st)^*}\cap X_{t^*} \overset{\ref{pglob}}{=} X_{(st)^*(st)}\cap X_{t^*t} \overset{\ref{ef_domain}}{=} X_{(st)^*(st)t^*t}=X_{(st)^*(st)}=X_{(st)^*},$$
which means that the domain of $\theta_s\circ\theta_t$ coincides with the domain of $\theta_{st}$ and, hence, concluding \ref{eglob}.

Conversely, if $\theta$ is partial action that satisfies \ref{eglob}, then for any $s \in \Se$ we would have $\theta_{ss^*}=\theta_s\circ\theta_{s^*}=\id_{X_s}$ and, consequently, $X_{ss^*}=X_s$ concluding \ref{pglob}.
\end{proof}

We end this section with two examples of partial actions that we shall explore again in the next section.

\begin{example}\label{ex_1_pt2}
Let $X=\{1,2,3,4\}$ and let $\Se=\{a,a^*,b,b^*,a^*a,aa^*,b^*b,bb^*\}$ be the inverse semigroupoid presented in Example \ref{ex_pt1}. Consider the following family of subsets of $X$:
\begin{multicols}{4}
$X_{b^*}=\{1,2\}$,

$X_{b^*b}=\{1,2\}$,

$X_b=\{1,4\}$,

$X_{bb^*}=\{1,4\}$,

$X_{a^*}=\{1\}$

$X_{a^*a}= \{1\}$,

$X_a=\{3\}$,

$X_{aa^*}=\{3,4\}$,
\end{multicols}
and the following family of bijections:
$$\begin{array}{cccc}
\theta_b \colon & X_{b^*} & \rightarrow & X_b \\
& 1 & \mapsto & 1\\
& 2 & \mapsto & 4
\end{array},
\qquad
\begin{array}{cccc}
\theta_{a}\colon & X_{a^*} & \rightarrow & X_a \\
& 1 & \mapsto & 4\\
& & &
\end{array},$$
$\theta_{a^*}=\theta_a^{-1}$, $\theta_{b^*}=\theta_b^{-1}$ and $\theta_{e}=\id_{X_e}$ for $e \in E(\Se)=\{b^*b, bb^*,a^*a, aa^*\}$. It is routine to check that $\theta=(\{X_s\}_{s\in\Se}, \{\theta_s\}_{s\in\Se})$ is, indeed, a partial action of $\Se$ on $X$.
\end{example}

\begin{example}\label{ex_2_pt2}
Let $Y=\{1,2,3\}$ and let $\Se=\{a,a^*,b,b^*,a^*a,aa^*,b^*b,bb^*\}$ be the inverse semigroupoid presented in Example \ref{ex_pt1}. Consider the following family of bijections:
$$\begin{array}{cccc}
\theta^Y_a \colon & Y & \rightarrow & Y \\
& 1 & \mapsto & 2 \\
& 2 & \mapsto & 3 \\
& 3 & \mapsto & 1
\end{array},$$
$\theta^Y_{a^*}={\theta^Y_a}^{-1}$ and $\theta^Y_s=\id_Y$ if $s \neq a,a^*$. Setting $Y_s=Y$ for every $s \in \Se$, it is routine to check that $\theta^Y=(\{Y_s\}_{s\in\Se}, \{\theta^Y_s\}_{s\in\Se})$ is a global action of $\Se$ on $Y$.
\end{example}

\section{Restriction and Globalization of partial actions}

It is well known that we can produce interesting examples of partial actions via a process of restriction of global actions. Let us explain this process in the present context: suppose $\theta=(\{Y_s\}_{s\in\Se}, \{\theta_s\}_{s\in\Se})$ is a partial action of an inverse semigroupoid $\Se$ on a set $Y$. If $X$ is a subset of $Y$, we may try to restrict the action of $\Se$ to $X$ by restricting the domain of each $\theta_s$ to $X \cap Y_{s^*}$. The main problem in this process is that, even though $x \in X\cap Y_{s^*}$, we may have $\theta_s(x) \notin X$.

However, we can turn around this problem by restricting each $\theta_s$ to the set formed by all $x \in X\cap Y_{s^*}$ such that $\theta_s(x)$ also lies in $X$. Explicitly, setting $X_s=\theta_s(X \cap Y_{s^*})\cap X$ and noticing that $\theta_s(X_{s^*})=X_s$, we may let
$${(\restr{\theta}{X})}_s\colon X_{s^*}\to X_s$$
to be the restriction of $\theta_s$ to $X_{s^*}$ onto $X_s$. It is then easy to verify that
$$\restr{\theta}{X}=(\{X_s\}_{s\in\Se}, \{ {(\restr{\theta}{X})}_s\}_{s\in\Se})$$
satisfies the axioms of Definition \ref{our_def} and, hence, is a partial action of $\Se$ on $X$ which, from now on, we shall call the \emph{restriction} of the action $\theta$ to $X$.

Let us point out that, even if the original action $\theta$ of $\Se$ on $Y$ is global, there is no guarantee that the restriction $\restr{\theta}{X}$ is a global action. Indeed, for any $x\in X\cap Y_{s^*}$ such that $\theta_s(x)$ does not lie in $X$, we would have $x \in X_{s^*s}$, but $x \notin X_{s^*}$.

The observation in the last paragraph raises an interesting problem: if we start with a partial action $\theta$ of an inverse semigroupoid $\Se$ on a set $X$, does there exist a set $Y$ containing $X$ and a global action $\eta$ of $\Se$ on $Y$ such that $\theta$ may be obtained by the restriction of the action of $\eta$ to $X$? The main goal of this work is to give an affirmative answer for this question.

Before we proceed, let us introduce the appropriate morphisms to form the category of partial actions. 

\begin{definition}
Let $\Se$ be an inverse semigroupoid and let $\theta^X=(\{X_s\}_{s\in\Se}, \{\theta_s^X\}_{s\in\Se})$ and $\theta^Y=(\{Y_s\}_{s\in\Se}, \{\theta_s^Y\}_{s\in\Se})$ be partial actions of $\Se$ on sets $X$ and $Y$, respectively. A function $\varphi\colon X\rightarrow Y$ is said to be an $\Se$\emph{-function} if $\varphi(X_s)\subseteq Y_s$ for every $s\in\Se$ and, moreover, $\varphi\big(\theta_s^X(x)\big)=\theta_s^Y\big(\varphi(x)\big)$ for every $x \in X_{s^*}$.
\end{definition}

Given an inverse semigroupoid $\Se$, the class of $\Se$-sets together with the class of $\Se$-functions form a category, which we denote by $\mathscr{A}_p(\Se)$. Likewise, the class of global $\Se$-sets together with the class of $\Se$-functions between then form a full subcategory of $\mathscr{A}_p(\Se)$, which we denote by $\mathscr{A}(\Se)$.

For example, if $(Y,\theta)$ is an $\Se$-set and $X\subseteq Y$, it is easy to verify that the inclusion map $i\colon X \to Y$ is an $\Se$-function between the $\Se$-sets $(X,\restr{\theta}{X})$ and $(Y,\theta)$. Actually, $i$ is even more than an injective $\Se$-function, it is an embedding in $\mathscr{A}_p(\Se)$ as follows.

\begin{definition}\label{embedding}
Let $\Se$ be an inverse semigroupoid and let $\theta^X=(\{X_s\}_{s\in\Se}, \{\theta_s^X\}_{s\in\Se})$ and $\theta^Y=(\{Y_s\}_{s\in\Se}, \{\theta_s^Y\}_{s\in\Se})$ be partial actions of $\Se$ on sets $X$ and $Y$, respectively. An injective $\Se$-function $\varphi\colon X\rightarrow Y$ is said to be an \emph{embedding} if
$$X_s=\varphi^{-1}\big(\theta_s^Y(\varphi(X)\cap Y_{s^*})\big)$$
for every $s \in \Se$.
\end{definition}

Notice that, if $\varphi\colon X\rightarrow Y$ is an embedding as in the Definition above, then the action $\theta^Y$ of $\Se$ on $Y$ induces the action $\theta^X$ of $\Se$ on $X$ in the sense that, for given $x_1,x_2 \in X$ and $s \in \Se$ such that $\varphi(x_1) \in Y_{s^*}$ and $\theta_s(\varphi(x_1))=\varphi(x_2)$, then $x_1 \in X_{s^*}$ and $\theta_s(x_1)=x_2$.

Actually, in order to an injective function $\varphi$ as in Definition \ref{embedding} be an embedding is necessary, and sufficient that, for given $x \in X$ and $s \in \Se$, we have that $x \in X_{s^*}$ if, and only if, $\varphi(x)\in Y_{s^*}$ and $\theta^Y_s(\varphi(x)) \in \varphi(X)$ and, in the affirmative case, $\theta^Y_s(\varphi(x))=\varphi(\theta_s^X(x))$. This amounts to say that $\varphi$ co-restricts to an isomorphism $\varphi\colon X \to \varphi(X)$ between $(X,\theta^X)$ and $(\varphi(X),\restr{\theta^Y}{\varphi(X)})$ in $\mathscr{A}_p(\Se)$.

This discussion gives the appropriate definition for globalizations of partial actions.

\begin{definition}\label{globalization}
Let $(X,\theta^X)$ and $(Y,\theta^Y)$ be $\Se$-sets and $\varphi \colon X\to Y$ a $\Se$-function. The triple $(\varphi,Y,\theta^Y)$ is said to be 
\begin{enumerate}[label=(\roman*)]
\item a \emph{globalization} of $(X,\theta^X)$ if $\varphi$ is an embedding and $(Y,\theta^Y)$ is a global $\Se$-set.
\item an \emph{universal globalization} of $(X,\theta^X)$ if it is a globalization of $(X,\theta^X)$ and, for any globalization $(\psi,Z,\theta^Z)$ of $(X,\theta^X)$, there exists a unique \emph{mediating} $\Se$-function $\sigma\colon Y\to Z$ such that the diagram
$$\xymatrix{ X \ar[rr]^{\varphi} \ar[rd]_{\psi}& & Y \ar@{..>}[ld]^{\sigma}\\ & Z & }$$
commutes. 
\end{enumerate}
\end{definition}

In the context of Definition \ref{globalization} above, if the $\Se$-function $\varphi$ is implicit in the context, we may say simply that $(Y,\theta^Y)$ is a globalization of $(X,\theta^X)$.

We shall prove that every partial action admits a universal globalization. Moreover, we shall prove that our construction gives a reflector from $\mathscr{A}_p(\Se)$ on $\mathscr{A}(\Se)$ (Theorem \ref{o_teorema}), whose definition we recall below.

\begin{definition}\label{reflector}
A full subcategory $\mathcal{D}$ of a category $\mathcal{C}$ is said to be \emph{reflective} if, for any $\mathcal{C}$-object $C$, there exists a $\mathcal{D}$-object $R(C)$ and a $\mathcal{C}$-morphism $\varphi_C\colon C \to R(C)$ such that, for every $\mathcal{D}$-object $D$ and $\mathcal{C}$-morphism $\psi\colon C \to D$, there exists a unique mediating $\mathcal{D}$-morphism $\sigma\colon R(C) \to D$ such that the diagram
$$\xymatrix{C \ar[rr]^{\varphi_C} \ar[rd]_{\psi}& & R(C) \ar@{..>}[ld]^{\sigma}\\ & D & }$$
commutes. Moreover, in this case, the pair $(\varphi_C,R(C))$ is said to be a $\mathcal{D}$-\emph{reflector} of $C$.
\end{definition}

In the context of Definition \ref{reflector} above, if the morphism $\varphi_C$ is implicit in the context, we may say simply that $R(C)$ is a $\mathcal{D}$-reflector of $C$.

If $\mathcal{D}$ is a reflective subcategory of a category $\mathcal{C}$, then the map $C \mapsto R(C)$ induces a \emph{reflector} functor $R\colon \mathcal{C} \to \mathcal{D}$ which is right adjoint to the inclusion functor $\mathcal{D} \to \mathcal{C}$.

\

\textbf{From now on we fix an inverse semigroupoid \boldmath{$\Se$}, a set \boldmath{$X$} and a partial action
\boldmath{$$\theta=(\{X_s\}_{s\in\Se},\{\theta_s\}_{s\in\Se}).$$}}

\

Our goal now is to construct a global $\Se$-set $(E,\eta)$ and an $\Se$-function $i\colon X \to E$ such that $(i,E,\eta)$ is a globalization of $(X,\theta)$ which is a $\mathscr{A}(\Se)$-reflector of $(X,\theta)$. Notice that, a globalization of $(X,\theta)$ which is also a $\mathscr{A}(\Se)$-reflector of $(X,\theta)$ is, automatically, a universal globalization of $(X,\theta)$. 

For reference, our construction is mainly inspired in \cite{Patrik}. We begin defining a relation on a subset of $\Se\times X$ as follows.

\begin{definition}\label{relation}
   Let $D:=\{(s,x)\in \Se\times X : x \in X_{s^*s}\}$. Define a relation $\sim$ on $D$ such that, for all $(s,x),(t,y)\in D$, $(s,x)\sim (t,y)$ if either
\begin{enumerate}[label=(R\arabic*)]
    \item\label{r1} $(t^*,s)\in\Se^{(2)}$, $x \in X_{s^*t}$ and $\theta_{t^*s}(x)=y$;
    \item\label{r2} or $s, t \in E(\Se)$ and $x=y$.
\end{enumerate}
\end{definition}

Now, we point some observations about the relation $\sim$ defined above. 

\begin{remark}\label{obs2}
First, in the case in which item \ref{r2} is checked, we have $x\in X_s$ and $y\in X_t$ and $\theta_s(x)=x=y=\theta_t(y)$, by \ref{p1}.

Still in the case in which item \ref{r2} is checked, if $s$ and $t$ are composable, then by Proposition \ref{ef_domain}, $X_{st}=X_s\cap X_t$ and, hence, $(t,s)\in\Se^{(2)}$, $x \in X_{st}$ and $\theta_{st}(x)=x$ by \ref{p1}. This means that $(s,x)\sim (t,y)$ is also checked by item \ref{r1}. In other words, item \ref{r2} in Definition \ref{relation} plays a relevant role only in the case $s$ and $t$ are idempotent elements over different objects.
\end{remark}

The set $E$ that we are looking for is going to be the quotient of $D$ by an equivalence relation which identifies pairs satisfying \ref{r1} or \ref{r2}. Unfortunately, the relation $\sim$ in Definition \ref{relation} may fail to be an equivalence relation. Indeed, in Example \ref{ex_1_pt2}, we may notice that $(a,1)\sim(aa^*,4)$, $(aa^*,4)\sim(bb^*,4)$ but $(a,1)\nsim(bb^*,4)$. At least, $\sim$ is reflexive and symmetric.  

\begin{proposition}\label{quasi_equivalence}
The relation $\sim$ defined in \ref{relation} is reflexive and symmetric. 
\end{proposition}

\begin{proof}
Let $(s,x)\in D$. Since $x \in X_{s^*s}$ and $s^*s \in E(\Se)$, we have $\theta_{s^*s}(x)=x$ by \ref{p1}. This means that $(s,x)\sim (s,x)$ by item \ref{r1} of Definition \ref{relation}. Hence, $\sim$ is indeed reflexive.

It remains to show that $\sim$ is symmetric. Let $(s,x), (t,y) \in D$ and suppose that $(s,x)\sim (t,y)$. If $x=y$ and $s,t\in E(\Se)$ it is clear that $(t,y)\sim (s,x)$. Otherwise, assume that $(t^*,s)\in\Se^{(2)}$, $x \in X_{s^*t}$ and $\theta_{t^*s}(x)=y$. Then, we must have $(s^*,t)\in\Se^{(2)}$, $y \in X_{t^*s}$ and $\theta_{s^*t}(y)=x$, which means that $(t,y)\sim (s,x)$ by \ref{r2}.
\end{proof}

Since $\sim$ may not be an equivalence relation, We turn around this problem considering the equivalence relation generated by $\sim$, as follows.

\begin{definition}\label{releq}
Let $\approx$ be the smallest equivalence relation on $D$ containing the relation $\sim$ defined in \ref{relation}. Explicitly, if $(s,x), (t,y) \in D$, then $(s,x)\approx (t,y)$ if, and only if, there exist $n\in \mathbb{N}$ and $(r_1,z_1),\ldots,(r_n,z_n)\in D$ such that $(r_1,z_1)=(s,x)$, $(r_n,z_n)=(t,y)$ and $(r_i,z_i)\sim (r_{i+1},z_{i+1})$ for all $i\in\{ 1,\ldots,n-1\}$. We denote by $E$ the quotient of $D$ by the equivalence relation $\approx$. The equivalence class of an element $(s,x)\in D$ is denoted by $[s,x]$.
\end{definition}

The next step is to define an action of $\Se$ on $E$. The idea here is that, if an element $p \in \Se$ acts in the class of an element $(s,x) \in D$ such that $(p,s)\in\Se^{(2)}$, we should obtain the class of the element $(ps,x)$. However, in order to $(ps,x)$ be an element of $D$ we must have $x \in X_{(ps)^*(ps)}$. So, we now prove some results with the aiming to obtain a well defined action of $\Se$ on $E$ in this fashion.

\begin{lemma}\label{lemaind}
Let $(s,x)$, $(t,y)\in D$ be such that $(s,x)\sim(t,y)$ and let $p\in \Se$ be such that $(p,s),(p,t)\in \Se^{(2)}$. Then $x \in X_{s^*p^*ps}$ if, and only if, $y \in X_{t^*p^*pt}$. Moreover, when the memberships are verified, we have $(ps,x)\sim(pt,y)$.
\end{lemma}

\begin{proof}
Notice, at first, that the membership $x\in X_{(ps)^*(ps)}=X_{s^*p^*ps}$ is necessary (and sufficient) to guarantee that $(ps,x)$ lies in $D$.

Let $(s,x)$, $(t,y)\in D$ be such that $(s,x)\sim(t,y)$ and let $p\in \Se$ be such that $(p,s),(p,t)\in \Se^{(2)}$. Notice that, if $s,t\in E(\Se)$, since $(p,s),(p,t)\in\Se^{(2)}$, we must have $(s,t),(t,s)\in\Se^{(2)}$. Hence, by the second paragraph of Remark \ref{obs1}, we may assume that $(s,x)\sim(t,y)$ is checked by item \ref{r1} of Definition \ref{relation}. Then, assume that $(t^*,s)\in\Se^{(2)}$, $x\in X_{s^*t}$ and $\theta_{t^*s}(x)=y$. By symmetry of $\sim$, we have $(s^*,t)\in\Se^{(2)}$, $y\in X_{t^*s}$ and $\theta_{s^*t}(y)=x$.

Assume now that $x \in X_{s^*p^*ps}$ and notice that $y$ must lie in the domain
$$\theta_{s^*t}^{-1}\big(X_{s^*t}\cap\theta_{s^*p^*ps}^{-1}(X_{s^*p^*ps}\cap X_{s^*t})\big)=\theta_{s^*t}^{-1}(X_{s^*t}\cap X_{s^*p^*ps})$$
of $\theta_{t^*s}\circ\theta_{s^*p^*ps}\circ\theta_{s^*t}$ since $x$ lies in both $X_{s^*t}$ and $X_{s^*p^*ps}$ and $\theta_{s^*t}(y)=x$. Now, observing that
$$(t^*s)(s^*p^*ps)(s^*t)=t^*ss^*p^*pss^*t=t^*ss^*p^*pt\leq t^*p^*pt$$
and using \ref{e2} and \ref{e3}, we have that $y\in X_{t^*ss^*p^*pt} \subseteq X_{t^*p^*pt}$ as desired.

Conversely, if $y \in X_{t^*p^*pt}$, just notice that
$$(s^*t)(t^*p^*pt)(t^*s)=s^*tt^*p^*ptt^*s=s^*tt^*p^*ps\leq s^*p^*ps$$
and use a similar argument to deduce that $x \in X_{s^*p^*ps}$.

Finally, we assume that $x \in X_{s^*p^*ps}$ and $y \in X_{t^*p^*pt}$ and prove that $(ps,x)\sim(pt,y)$. For that purpose, notice that $(t^*p^*,ps)\in\Se^{(2)}$ and that
$$(pt)^*(ps)=t^*p^*ps=t^*p^*pss^*s=t^*ss^*p^*ps=(t^*s)(s^*p^*ps),$$
from which we can deduce by \ref{p3} that
$$x \in X_{s^*t}\cap X_{s^*p^*ps}=\theta_{s^*p^*ps}^{-1}(X_{s^*t}\cap X_{s^*p*ps}) \overset{\ref{p3}}{=} X_{(t^*p^*ps)^*}\cap X_{s^*p^*ps} \subseteq X_{(ps)^*(pt)},$$
and, moreover, that
$$\theta_{(pt)^*(ps)}(x)=\theta_{t^*s}\big(\theta_{s^*p^*ps}(x)\big)=\theta_{t^*s}(x)=y.$$
Therefore, by item \ref{r1} of Definition \ref{relation}, $(ps,x)\sim(pt,y)$ as desired.
\end{proof}

We may also slightly modify Lemma \ref{lemaind} by changing the relation $\sim$ by the equivalence relation $\approx$ generated by $\sim$. But first, we will need an auxiliary result.

\begin{lemma}\label{lema_acao}
Let $(s,x)$, $(t,y)\in D$ such that $(s,x)\approx(t,y)$. Then $x\in X_{s^*}$ if, and only if, $y\in X_{t^*}$. Moreover, when the memberships are verified, we have $\theta_s(x)=\theta_t(y)$.
\end{lemma}

\begin{proof}
Let $(s,x)$, $(t,y)\in D$ and assume, initially, that $(s,x)\sim(t,y)$. If $(s,x)\sim(t,y)$ by \ref{r2} of Definition \ref{relation}, then the result follows direct from the comment in the first paragraph after Definition \ref{relation}. Otherwise, $(s,x)\sim(t,y)$ by \ref{r1} of Definition \ref{relation} and, then, we must have $(t^*,s)\in\Se^{(2)}$, $x \in X_{s^*t}$ and $\theta_{t^*s}(x)=y$.

If $x \in X_{s^*}$, then $x \in X_{s^*}\cap X_{s^*t}$ which is the domain of $\theta_{t^*}\circ\theta_s$, by \ref{p3}. Hence, again by \ref{p3} we have that
$$y=\theta_{t^*s}(x)=\theta_{t^*}\big(\theta_s(x)\big)$$
which lies in $X_{t^*}$. Moreover, applying $\theta_t$ in the above equality, we obtain that $\theta_t(y)=\theta_s(x)$.

If $y \in X_{t^*}$ we may use the symmetry of $\sim$ and a similar argument to argue that $x \in X_{s^*}$ as well.

For the general case, there exist $n \in \mathbb{N}$ and $(r_1,z_1), \ldots, (r_n,z_n) \in D$ such that $(r_1,z_1)=(s,x)$, $(r_n,z_n)=(t,y)$ and $(r_i,z_i)\sim(r_{i+1},z_{i+1})$ for all $i \in \{1,\ldots, n-1\}$. Notice that, by the previous argument, $z_i\in X_{r_i^*}$ if, and only if, $z_{i+1} \in X_{r_{i+1}^*}$ for all $i\in\{1\ldots,n-1\}$ and, hence, $x\in X_{s^*}$ if, and only if, $y\in X_{t^*}$. Moreover, assuming that $x\in X_{s^*}$ and $y\in X_{t^*}$, then $z_i \in X_{r_i^*}$ for all $i\in\{1\ldots,n\}$ and, again by the previous argument, $\theta_{r_i}(z_i)=\theta_{r_{i+1}}(z_{i+1})$ for all $i\in\{1\ldots,n-1\}$, which allow us to deduce that $\theta_s(x)=\theta_t(y)$ as stated.
\end{proof}

\begin{lemma}\label{lemaind2}
Let $(s,x)$, $(t,y)\in D$ be such that $(s,x)\approx(t,y)$ and let $p\in \Se$ be such that $(p,s),(p,t)\in \Se^{(2)}$. Then $x \in X_{s^*p^*ps}$ if, and only if, $y \in X_{t^*p^*pt}$. Moreover, when the memberships are verified, we have $(ps,x)\approx(pt,y)$.
\end{lemma}

\begin{proof}
Let $(s,x)$, $(t,y)\in D$ be such that $(s,x)\approx(t,y)$ and let $p\in \Se$ be such that $(p,s),(p,t)\in \Se^{(2)}$.

Assume initially that $s\in E(\Se)$. Then, since $(s,x)\in D$, we have $x \in X_{s^*}$. By Lemma \ref{lema_acao} and \ref{p1}, we have that $y \in X_{t^*}$ and $\theta_t(y)=\theta_s(x)=x$. Since $(tt^*,t)\in\Se^{(2)}$, we deduce from the previous sentence that
\begin{equation}\label{lemaind2_eq4}
(t,y)\sim(tt^*,x)
\end{equation}
by item \ref{r1} of Definition \ref{relation}. Notice that, indeed, $(tt^*,x)\in D$, since $x=\theta_t(y)\in X_t$ and $X_t\subseteq X_{tt^*}$ by \ref{p2}.

Since $tt^*$ and $s$ are idempotent elements, we have that
\begin{equation}\label{lemaind2_eq5}
(tt^*,x)\sim(s,x)
\end{equation}
by item \ref{r2} of Definition \ref{relation}. 

Combining (\ref{lemaind2_eq4}) and (\ref{lemaind2_eq5}) with Lemma \ref{lemaind}, we deduce that $x \in X_{s^*p^*ps}$ if, and only if, $y \in X_{t^*p^*pt}$ and, in the affirmative case, that $(ps,x)\sim(ptt^*,x)\sim(pt,y)$.

We can use a similar argument if $t \in E(\Se)$. Hence, we now assume that neither $s \in E(\Se)$, nor $t \in E(\Se)$. Then, there exist $n\in\mathbb{N}$ and $(r_1,z_1), \ldots, (r_n,z_n)\in D$ such that $(r_1,z_1)=(s,x)$, $(r_n,z_n)=(t,y)$ and $(r_i,z_i)\sim(r_{i+1},z_{i+1})$ for all $i\in\{1,\ldots,n-1\}$.

Assume $x \in X_{s^*p^*ps}$. We proceed by induction over $n$. If $n=2$, Lemma \ref{lemaind} applies directly. If $n>2$, assume that our claim holds for $m<n$. Since $s \notin E(\Se)$, we must have $(s,x)\sim(r_2,z_2)$ by item \ref{r1} of Definition \ref{relation} and, hence, $(r_2^*,s)\in\Se^{(2)}$. This implies that $(p,r_2)\in\Se^{(2)}$ and, then, Lemma \ref{lemaind} applies. Hence $z_2 \in X_{r_2^*p^*pr_2}$ and $(ps,x)\sim(pr_2,z_2)$. If $r_2\in E(\Se)$, then $(r_2,z_2)\approx(t,y)$ and we can use the initial argument to deduce that $y\in X_{t^*p^*pt}$ and $(ps,x)\sim(pr_2,z_2)\approx(pt,y)$. If $r_2\notin E(\Se)$, then we can use the induction hypothesis to infer that $y\in X_{t^*p^*pt}$ and $(ps,x)\sim(pr_2,z_2)\approx(pt,y)$. In any such case, we obtain the desired result.

Assuming that $y\in X_{t^*p^*pt}$, we can use a similar argument (this time from right to left) to establish that $x \in X_{s^*p^*ps}$ and $(ps,x)\approx(pt,y)$ as desired.
\end{proof}

We now begin to construct an action of $\Se$ on $E$ by first defining a collection of functions between subsets of $D$. For this, given $s\in\Se$, we set
\begin{align}\label{zeta_s}
\zeta_s\colon \quad D_{s^*} & \longrightarrow D_s\\
(p,x) &\longmapsto (sp,x), \nonumber
\end{align}
in which
\begin{align}\label{D_s}
D_{s^*}=\{(p,x) \in D : (s,p)\in\Se^{(2)} \text{ and } x \in X_{p^*s^*sp}\}.
\end{align}

Notice that, $\zeta_s$ is well defined since $(s^*,sp)\in\Se^{(2)}$ and $X_{(s^*sp)^*(s^*sp)}=X_{p^*s^*sp}$, which implies that $(sp,x)$ indeed lies in $D_s$.

Moreover, if $(p,x), (q,y)\in D_{s^*}$ are such that $(p,x)\approx(q,y)$, then, by Lemma \ref{lemaind2}, $\zeta_s(p,x)=(sp,x)\approx(sq,y)=\zeta_s(q,y)$. Thus, setting $E_s$\label{E_s} as the image of $D_s$ by the quotient map of $D$ by the equivalence relation $\approx$, we obtain a well defined function $\eta_s\colon E_{s^*} \to E_s$ such that
$$\eta_s\big([p,x])=[sp,x]$$
for any $(p,x)\in D_{s^*}$. We, thus, obtain a pair
\begin{equation}\label{parfun}
    \eta = (\{E_s\}_{s\in\Se},\{\eta_s\}_{s\in\Se})
\end{equation}
which we now prove that is a global action of $\Se$ on $E$. Before we proceed, just to emphasize, a point in $E$ lies in $E_s$ if, and only if, it admits a representative in $D_s$.

\begin{theorem}\label{global_action}
The pair $\eta = (\{E_s\}_{s\in\Se},\{\eta_s\}_{s\in\Se})$, as in (\ref{parfun}), is a global action of $\Se$ on $E$.  
\end{theorem}

\begin{proof}
We shall argue \ref{e1}, \ref{e2}, \ref{e3} and \ref{eglob} to use Proposition \ref{def_lui}.

Let $s \in \Se$ and $[p,x]\in E_{s^*}$. Since $[p,x]\in E_{s^*}$, there exist $(q,y)\in D_{s^*}$ such that $(p,x)\approx(q,y)$. We have already argued that $\zeta_s(q,y)=(sq,y)\in D_s$. Moreover, since $(q^*,s^*sq)\in\Se^{(2)}$, $y\in X_{q^*s^*sq}$ and $\theta_{q^*s^*sq}(y)=y$, we have that $(s^*sq,y)\sim(q,y)$ by item \ref{r1} of Definition \ref{relation}. Therefore,
$$\eta_{s^*}\big(\eta_s([p,x])\big)=\eta_{s^*}\big(\eta_s([q,y])\big)=\eta_{s^*}([sq,y])=[s^*sq,y]=[q,y]=[p,x],$$
from where we infer that $\eta_{s^*}\circ\eta_s=\id_{E_{s^*}}$. Similarly, we obtain that $\eta_s\circ\eta_{s^*}=\id_{E_s}$ and, thus, $\eta_s^{-1}=\eta_{s^*}$.

Moreover, let $[p,x]\in E$. This means that $x\in X_{p^*p}$ and, since, $(p^*,p)\in\Se^{(2)}$, we have that $[p,x] \in E_p$. This completes the proof that $E=\cup_{s\in\Se}E_s$ and item \ref{e1} is verified.

Next, we prove that $\eta_s\circ\eta_t=\eta_{st}$ to verify \ref{e2} and \ref{eglob} simultaneously. Let $(s,t)\in\Se^{(2)}$. In one hand, let $[p,x]\in E_{(st)^*}$. Hence, there exists $(q,y)\in D_{(st)^*}$ such that $(p,x)\approx (q,y)$. Since $(st,q)\in\Se^{(2)}$ and $y\in X_{q^*t^*s^*stq}$, we have that $(s,tq)\in\Se^{(2)}$ and, thus, $(tq,y)\in D_{s^*}$. Moreover, $(t,q)\in\Se^{(2)}$ and $y \in X_{q^*t^*tq}$, since $X_{q^*t^*s^*stq}\subseteq X_{q^*t^*tq}$ by Proposition \ref{extension}, from where we deduce that $(q,y)\in D_{t^*}$. Therefore, $[q,y]\in E_{t^*}$ and $\eta_t([q,y])=[tq,y]\in E_{s^*}$ which allow us to infer that
$$[p,x]=[q,y]\in \eta_t^{-1}(E_t \cap E_{s^*})$$
and
$$\eta_{st}([p,x])=\eta_{st}([q,y])=[stq,y]=\eta_s([tq,y])=\eta_s\big(\eta_t([q,y])\big)=\eta_s\big(\eta_t([p,x])\big).$$

On the other hand, let $[p,x]\in \eta_t^{-1}(E_t \cap E_{s^*})$. Thus, $[p,x]\in E_{t^*}$ and $\eta_t([p,x])\in E_{s^*}$, from where we obtain $(q,y)\in D_{t^*}$ and $(r,z)\in D_{s^*}$ such that $(p,x)\approx(q,y)$ and $\zeta_t(q,y)\approx(r,z)$. Since $(r,z)\in D_{s^*}$ and $(s,tq)\in\Se^{(2)}$, by Lemma \ref{lemaind2}, we have $(tq,y)=\zeta_t(q,y)\in D_{s^*}$. This means that $y\in X_{q^*t^*s^*stq}$ and, thus, $(q,y)\in D_{(st)^*}$. Hence, $[p,x]=[q,y]\in E_{(st)^*}$ and
$$\eta_s\big(\eta_t([p,x])\big)=\eta_s\big(\eta_t([q,y])\big)=\eta_s([tq,y])=[stq,y]=\eta_{st}([q,y])=\eta_{st}([p,x]).$$
This completes the proof that $\eta_s\circ\eta_t=\eta_{st}$ and, hence, \ref{e2} and \ref{eglob} are verified.

It remains to verify \ref{e3}. For that purpose, let $s,t \in \Se$ such that $s\leq t$ and let $[p,x]\in E_s$. Thus, there exist $(q,y)\in D_s$ such that $(p,x)\approx (q,y)$. Since, $(q,y)\in D_s$, we have that $(s^*,q)\in\Se^{(2)}$ and $y\in X_{q^*ss^*q}$. Since $s\leq t$, we have that $s^*\leq t^*$ and, hence, $(t^*,q)\in\Se^{(2)}$ and $y\in X_{q^*ss^*q}\subseteq X_{q^*tt^*q}$, by Proposition \ref{extension}. This means that $(q,y)\in D_t$ and, hence, $[p,x]=[q,y]\in E_t$. This verifies \ref{e3} and closes this proof.
\end{proof}

We now prove that $(E,\eta)$ is a $\mathscr{A}(\Se)$-reflector of $(X,\theta)$.

\begin{theorem}\label{o_teorema}
Let $\Se$ be an inverse semigroupoid. The subcategory $\mathscr{A}(\Se)$ is reflective in $\mathscr{A}_p(\Se)$. More precisely, for any $\Se$-set $(X,\theta)$, there exist a global $\Se$-set $(E,\eta)$ and an $\Se$-function $i\colon X\to E$ such that $(i,E,\eta)$ is a $\mathscr{A}(\Se)$-reflector of $(X,\theta)$.
\end{theorem}

\begin{proof}
By \ref{p1}, for all $x\in X$, there exists $e\in E(\Se)$ such that $x\in X_e$. So, define $i\colon X\to E$ by $x\mapsto [e,x]$. Notice that $i$ is well defined. In fact, if $e,f\in E(\Se)$ are such that $x \in X_e$ and $x \in X_f$, by item \ref{r2} of Definition \ref{relation} we have $(e,x)\sim (f,x)$ and so $[e,x]=[f,x]$. 

Moreover, $i$ is $\Se$-function. In fact, suppose that $x \in X_{s^*}$. Then, by \ref{p2}, $x \in X_{s^*s}$ and, thus, $i(x)=[s^*s,x]$. Since $(s,s^*s)\in \Se^{(2)}$ and $x\in X_{s^*s}$, we have $(s^*s,x) \in D_{s^*}$ and $\eta_s([s^*s,x])=[s,x]$. On the other hand, since $\theta_s(x)\in X_s \subseteq X_{ss^*}$, then $(ss^*,\theta_s(x))\in D$. That said, notice that  $(s,x)\sim (ss^*,\theta_s(x))$ and so $[s,x]=[ss^*,\theta_s(x)]$. Finally, 
$$\eta_s\big(i(x)\big)=\eta_s([s^*s,x])=[s,x]=[ss^*,\theta_s(x)]=i(\theta_s(x)).$$

Next we prove that $(i,E,\eta)$ is a $\mathscr{A}(\Se)$-reflector of $(X,\theta)$. Indeed, we already know by Theorem \ref{global_action} that $(E,\eta)$ is a global $\Se$-set. So, let $(Z,\omega)$ be a global $\Se$-set and $j\colon X\to Z$ be an $\Se$-function and set
\begin{align*}
\sigma_D\colon \qquad D & \longrightarrow Z\\
(s,x) &\longmapsto \omega_s(j(x)).
\end{align*}

Notice that $\sigma_D$ is a well defined function. Indeed, given $(s,x)\in D$, we have $x \in X_{s^*s}$. Since $j$ is an $\Se$-function, we have $j(x)\in j(X_{s^*s})\subseteq Z_{s^*s}$. Since $\omega$ is a global action of $\Se$ on $Z$, we have that $j(x) \in Z_{s^*}$ by \ref{pglob}.

Furthermore, $\sigma_D$ factors trough $E$. As a matter of fact, it is enough to prove that $\sigma_D(s,x)=\sigma_D(t,y)$ for $(s,x),(t,y)\in D$ such that $(s,x)\sim(t,y)$. If $(s,x)\sim(t,y)$ by \ref{r1} of Definition \ref{relation}, then $(t^*,s) \in\Se^{(2)}$, $x \in X_{s^*t}$ and $\theta_{t^*s}(x)=y$. Since $x \in X_{s^*t}$ and $j$ is an $\Se$-function, we have that $j(x)\in j(X_{s^*t})\subseteq Z_{s^*t}$ and $\omega_{t^*s}(j(x))=j(\theta_{t^*s}(x))=j(y)$. Since $\omega$ is a global action of $\Se$ on $Z$, by \ref{eglob}, we have $\omega_{t^*s}=\omega_{t^*}\circ\omega_s$ and, hence, $\omega_s(j(x))=\omega_t(j(y))$. We, thus, have
$$\sigma_D(s,x)=\omega_s(j(x))=\omega_t(j(y))=\sigma_D(t,y).$$
If $(s,x)\sim(t,y)$ by \ref{r2} of Definition \ref{relation}, then $s,t \in E(\Se)$ and $x=y$. By \ref{p1}, we have
$$\sigma_D(s,x)=\omega_s(j(x))=j(x)=j(y)=\omega_t(j(y))=\sigma_D(t,y).$$
Therefore, we obtain a well-defined function
\begin{align*}
\sigma\colon \qquad E & \longrightarrow Z\\
[s,x] &\longmapsto \omega_s(j(x)).
\end{align*}

We claim that $\sigma$ is an $\Se$-function. In fact, let $s\in\Se$ and let $[p,x] \in E_{s^*}$. Thus, there exists $(q,y)\in D_{s^*}$ such that $(p,x)\approx(q,y)$. Since $(q,y)\in D_{s^*}$, we have that $(s,q)\in\Se^{(2)}$ and $y \in X_{q^*s^*sq}$ and, moreover, $\eta_s([p,x])=\eta_s([q,y])=[sq,y]$. Since $j$ is an $\Se$-function, we have that $j(y)\in Z_{q^*s^*sq}$. By \ref{pglob}, $Z_{q^*s^*sq}=Z_{q^*s^*}$ which is the domain of $\omega_{sq}=\omega_s\circ\omega_q$ by \ref{eglob}. Hence,
$$\sigma([p,x])=\sigma([q,y])=\omega_q(j(y)) \in Z_{s^*}.$$
Moreover,
$$\omega_s\big(\sigma([p,x])\big)= \omega_s\big(\omega_q(j(y))\big)= \omega_{sq}(j(y))= \sigma([sq,y])= \sigma\big(\eta_s([p,x])\big),$$
and $\sigma$ is indeed an $\Se$-function.

Now, given $x \in X$, by \ref{p1}, choose $e \in E(\Se)$ such that $x \in X_e$. Notice that
$$\big(\sigma\circ i\big)(x)=\sigma([e,x])=\omega_e(j(x))=j(x),$$
which proves that the diagram
$$\xymatrix{ X \ar[rr]^{i} \ar[rd]_{j}& & E \ar@{..>}[ld]^{\sigma}\\ & Z & }$$
commutes. 

It remains to prove the uniqueness part. For that, assume that $\sigma':E\to Z$ is an $\Se$-function such that $\sigma'\circ i=j$. Let $[s,x]\in E$ and notice that $(s^*s,x)\in D_{s^*}$, since $(s,s^*s)\in\Se^{(2)}$ and $x\in X_{s^*s}$. Thus, $[s^*s,x]\in E_{s^*}$ and, being $\sigma'$ an $\Se$-function, we have that $\sigma'([s^*s,x])\in Z_{s^*}$ and moreover
$$\sigma([s,x])= \omega_s(j(x))= \omega_s\big((\sigma'\circ i)(x)\big)= \omega_s\big(\sigma'([s^*s,x])\big)= \sigma'\big(\eta_s([s^*s,x])\big)= \sigma'([s,x]),$$
from where we conclude that $\sigma'=\sigma$, as desired.
\end{proof}

We finally prove that $(E,\eta)$ is a universal globalization of $(X,\theta)$.

\begin{theorem}\label{o_teorema_2}
Let $\Se$ be an inverse semigroupoid. For every $\Se$-set $(X,\theta)$, there exist a global $\Se$-set $(E,\eta)$ and an $\Se$-function $i\colon X\to E$ such that $(i,E,\eta)$ is a universal globalization of $(X,\theta)$.
\end{theorem}

\begin{proof}
Let $(i,E,\eta)$ be as in the proof of Theorem \ref{o_teorema}. So, it only remains to prove that $i\colon X \to E$ is an embedding.

We first prove that $i$ is injective. Indeed, let $x,y\in X$ be such that $i(x)=i(y)$. Let $e,f\in E(S)$ such that $x \in X_e$ and $y\in X_f$ and, thus, we have $[e,x]=[f,y]$. Using Lemma \ref{lema_acao}, we infer that $x=\theta_e(x)=\theta_f(y)=y$.

Finally, we prove that the action $\eta$ of $\Se$ on $E$ induces the action $\theta$ of $\Se$ on $X$. Let $x,y \in X$ and $s\in \Se$ such that $i(x) \in E_{s^*}$ and $\eta_s(i(x))=i(y)$. We need to show that $x\in X_{s^*}$ and $\theta_s(x)=y$. Using \ref{p1}, choose $e,f\in E(\Se)$ such that $x \in X_e$ and $y \in X_f$. Since $[e,x]=i(x) \in E_{s^*}$, there exists $(r,z)\in D_{s^*}$ such that $(e,x)\approx(r,z)$. Hence,
$$[f,y]=i(y)=\eta_s(i(x))=\eta_s([e,x])=\eta_s([r,z])=[sr,z].$$
By Lemma \ref{lema_acao}, since $x \in X_e$, $y \in X_f$, $(e,x)\approx(r,z)$ and $(f,y)\approx(sr,z)$, we have that $z \in X_{(sr)^*} \cap X_{r^*}$, $\theta_r(z)=\theta_e(x)=x$ and $\theta_{sr}(z)=\theta_f(y)=y$. By \ref{p3}, $x=\theta_r(z)\in X_{s^*}$ and
$$\theta_s(x)=\theta_s(\theta_r(z))=\theta_{sr}(z)=y,$$
which finishes the proof.
\end{proof}

We apply now our construction in the two examples in the end of section 2.

\begin{example}\label{ex_1_pt3}
Let $\Se$ and $(X,\theta)$ be as in Example \ref{ex_1_pt2}. In this case, we have
\begin{align*}
D=\big\{(b,1),(b,2),(b^*b,1),(b^*b,2),(b^*,1),(b^*,4),(bb^*,1),(bb^*,4),\\(a,1),(a^*a,1),(a^*,3),(a^*,4),(aa^*,3),(aa^*,4)\big\}.
\end{align*}
and five equivalence classes with respect to relation $\approx$:
\begin{align*}
\overline{1} &=\{(b,1),(a^*,4),(b^*,1),(a^*a,1),(b^*b,1),(bb^*,1)\}, \\
\overline{2} & =\{(b^*b,2),(b^*,4)\}, \\
\overline{3} & =\{(a^*a,3)\}, \\
\overline{4} & =\{(b,2),(bb^*,4),(a,1),(aa^*,4)\}, \\
\overline{5} & =\{(a^*,3)\}.
\end{align*}
The family $\{D_s\}_{s \in \Se}$, as in Equation (\ref{D_s}) is given by
\begin{align*}
D_{a^*} & = D_{a^*a} = \{(a^*,3),(a^*,4),(b,1),(b^*,1),(a^*a,1),(bb^*,1),(b^*b,1)\}, \\
D_{a} & = D_{aa^*} = \{(a,1),(aa^*,3),(aa^*,4)\}, \\
D_{b^*} & = D_{b^*b} = \{(a^*,3),(a^*,4),(a^*a,1),(b,1),(b^*b,1),(b^*b,2),(b^*,1),(b^*,4),(bb^*,1)\}, \\
D_{b} & = D_{bb^*} = \{(a^*,3),(a^*,4),(b,1),(b,2),(b^*,1),(a^*a,1),(bb^*,1),(bb^*,4),(b^*b,1)\}.
\end{align*}
Thus, the globalization $(E,\eta)$ of $(X,\theta)$ is given by $\eta=(\{E_s\}_{s \in \Se},\{\eta_s\}_{s \in \Se})$ in which
\begin{align*}
& E_{a^*} = E_{a^*a} = \{\overline{1},\overline{5}\},
& E_{a} = E_{aa^*} = \{\overline{3},\overline{4}\}, \\
& E_{b^*} = E_{b^*b} = \{\overline{1},\overline{2},\overline{5}\},
& E_{b} = E_{bb^*} = \{\overline{1},\overline{4},\overline{5}\},
\end{align*}
$\eta_b\colon E_{b^*} \to E_b$ and $\eta_a\colon E_{a^*}\to E_a$ are given by
$$\begin{array}{rcl}
\eta_b(\overline{1}) & = & \eta_b([b^*b,1])=[b,1]=\overline{1} \\
\eta_b(\overline{2}) & = & \eta_b([b^*b,2])=[b,2]=\overline{4} \\
\eta_b(\overline{5}) & = & \eta_b([a^*,3])=[a^*,3]=\overline{5}
\end{array},
\qquad
\begin{array}{rcl}
\eta_a(\overline{1}) & = & \eta_a([a^*a,1])=[a,1]=\overline{4} \\
\eta_a(\overline{5}) & = & \eta_a([a^*,3])=[aa^*,3]=\overline{3} \\
& &
\end{array},$$
and $\eta_{a^*}=\eta_a^{-1}$, $\eta_{b^*}=\eta_b^{-1}$ and $\eta_{e}=\id_{E_e}$ for $e \in E(\Se)=\{b^*b, bb^*,a^*a, aa^*\}$. The $\Se$-function $i \colon X \to E$ is given by $i(1)=\overline{1}$, $i(2)=\overline{2}$, $i(3)=\overline{3}$ e $i(4)=\overline{4}$.
\end{example}

In the next example, we first restrict the global action of Example \ref{ex_2_pt2} and then compare the globalization of the restricted action with the original partial action.

\begin{example}\label{ex_2_pt3}
Let $\Se$ and $(Y,\theta^Y)$ be as in Example \ref{ex_2_pt2}. If we restrict the action $\theta^Y$ to $X=\{1,2\}$, we obtain a partial action $\theta=(\{X_s\}_{s \in \Se},\{\theta_s\}_{s\in \Se})$ of $\Se$ on $X$ such that $X_{a^*}=\{1\}$, $X_a=\{2\}$, $\theta_a\colon X_{a^*} \to X_a$ and $\theta_{a^*}\colon X_a \to X_{a^*}$ are evident, $X_s=X$ and $\theta_s=\id_X$ if $s \neq a,a^*$.

It is clear that $(j,Y,\theta^Y)$, in which $j\colon X \to Y$ is the inclusion map, is a globalization of $(X,\theta)$. However, this is not a universal globalization of $(X,\theta)$.

Indeed, applying our construction in this case, we obtain $D=\Se \times X$ and we have four equivalence classes with respect to relation $\approx$:
\begin{align*}
\overline{1} & =\{(b^*b,1),(bb^*,1),(a^*a,1),(aa^*,1),(b,1),(b^*,1),(a^*,2)\}, \\
\overline{2} & =\{(b^*b,2),(bb^*,2),(a^*a,2),(aa^*,2),(b,2),(b^*,2),(a,1)\}, \\
\overline{3} & =\{(a,2)\}, \\
\overline{4} & =\{(a^*,1)\}.
\end{align*}
The family $\{D_s\}_{s \in \Se}$, as in Equation (\ref{D_s}) is simply given by
$$D_{s^*}=\{(p,x)\in D : (s,p) \in\Se^{(2)} \}$$
and, thus, the globalization $(E,\eta)$ of $(X,\theta)$ is given by $\eta=(\{E_s\}_{s \in \Se},\{\eta_s\}_{s \in \Se})$ in which
\begin{align*}
& E_{a^*} = E_{a^*a} = \{\overline{1},\overline{2},\overline{4}\},
& E_{a} = E_{aa^*} = \{\overline{1},\overline{2},\overline{3}\}, \\
& E_{b^*} = E_{b^*b} = \{\overline{1},\overline{2},\overline{4}\},
& E_{b} = E_{bb^*} = \{\overline{1},\overline{2},\overline{4}\},
\end{align*}
$\eta_b\colon E_{b^*} \to E_b$ and $\eta_a\colon E_{a^*}\to E_a$ are given by
$$\begin{array}{rcl}
\eta_b(\overline{1}) & = & \eta_b([b^*b,1])=[b,1]=\overline{1} \\
\eta_b(\overline{2}) & = & \eta_b([b^*b,2])=[b,2]=\overline{2} \\
\eta_b(\overline{4}) & = & \eta_b([a^*,1])=[a^*,1]=\overline{4}
\end{array},
\qquad
\begin{array}{rcl}
\eta_a(\overline{1}) & = & \eta_a([a^*a,1])=[a,1]=\overline{2} \\
\eta_a(\overline{2}) & = & \eta_a([a^*a,2])=[a,2]=\overline{3} \\
\eta_a(\overline{4}) & = & \eta_a([a^*,1])=[aa^*,1]=\overline{1}
\end{array},$$
and $\eta_{a^*}=\eta_a^{-1}$, $\eta_{b^*}=\eta_b^{-1}$ and $\eta_{e}=\id_{E_e}$ for $e \in E(\Se)=\{b^*b, bb^*,a^*a, aa^*\}$. Moreover, the $\Se$-function $i\colon X \to E$ is given by $i(1)=\overline{1}$ and $i(2)=\overline{2}$.

Notice that, considering the globalization $(j,Y,\theta^Y)$, the mediating function $\sigma\colon E \to Y$ given by the universal property of $(i,E,\eta)$ is given by
$$\begin{array}{rcl}
\sigma(\overline{1}) & = & \sigma([b^*b,1]) = \theta_{b^*b}^Y(j(1)) = \theta_{b^*b}^Y(1) = 1 \\
\sigma(\overline{2}) & = & \sigma([b^*b,2]) = \theta_{b^*b}^Y(j(2)) = \theta_{b^*b}^Y(2) = 2 \\
\sigma(\overline{3}) & = & \sigma([a,2]) = \theta_a^Y(j(2)) = \theta_a^Y(2) = 3 \\
\sigma(\overline{4}) & = & \sigma([a^*,1]) = \theta_{a^*}^Y(j(1)) = \theta_{a^*}^Y(1) = 3
\end{array}$$
and hence, $\sigma$ is not injective. Moreover, the reader can check that any function $\sigma'\colon Y \to E$ such that $\sigma' \circ j=i$ is not an $\Se$-function.
\end{example}

The last example shows that the mediating function given by the universal property of the globalization $(E,\eta)$ may not be injective (and, hence, may not be an embedding). However, we can guarantee its injectivity in some especial subsets of $E$.

Indeed, for each $u \in \Se^{(0)}$, set
\begin{align}\label{D_u}
D_u=\{(s,x) \in D : c(s)=u\}.
\end{align}
and let $E_u$\label{E_u} be the image of $D_u$ by the quotient map of $D$ by the equivalence relation $\approx$.

We end this paper by showing that the mediating function given by the universal property of $(i,E,\eta)$ is injective in each $E_u$.

\begin{proposition}\label{a_proposicao}
Let $(X,\theta)$ be an $\Se$-set and $(j,Z,\omega)$ a globalization of $(X,\theta)$. Given any $u \in \Se^{(0)}$, the mediating function $\sigma\colon E \to Z$ given by the universal property of the globalization $(i,E,\eta)$ is injective in $E_u$.
\end{proposition}

\begin{proof}
Let $[s,x], [t,y] \in E_u$ such that $\sigma([s,x])=\sigma([t,y])$. With no loss in generality, we may assume $(s,x), (t,y) \in D_u$.  This means that
$$\omega_s(j(x))=\sigma([s,x])=\sigma([t,y])=\omega_t(j(y))$$
and, hence, $\omega_s(j(x))\in Z_t \cap Z_s$. Since $c(s)=c(t)=u$, we have $(t^*,s)\in S^{(2)}$ and, since $\omega$ is a global action, we have $\omega_{t^*s}=\omega_{t^*}\circ\omega_s$ by \ref{eglob}. Thus, $j(x) \in \omega_s^{-1}(Z_t \cap Z_s)=Z_{s^*t}$ and
$$\omega_{t^*s}(j(x))=\omega_{t^*}\big(\omega_s(j(x))\big)=\omega_{t^*}\big(\omega_t(j(y))\big)=j(y).$$

Since $j$ is an embedding, $x \in X_{s^*t}$ and $\theta_{t^*s}(x)=y$, which implies that $(s,x)\sim (t,y)$ by \ref{r2}.
\end{proof}

\bibliographystyle{acm}
\bibliography{references.bib}
\end{document}